\documentclass[11pt,letterpaper]{amsart}

\usepackage[margin=1in]{geometry}
\usepackage{amsmath,amssymb,amsthm}
\usepackage{hyperref}
\usepackage{orcidlink}

\newtheorem{theorem}{Theorem}
\newtheorem{lemma}[theorem]{Lemma}
\newtheorem{corollary}[theorem]{Corollary}
\theoremstyle{remark}
\newtheorem{remark}[theorem]{Remark}

\DeclareMathOperator{\parity}{par}
\DeclareMathOperator{\Bal}{Bal}
\DeclareMathOperator{\diag}{diag}
\DeclareMathOperator{\adj}{adj}
\numberwithin{equation}{section}

\title{Transpose Symmetry of Injectivity over\\Commutative Semirings}

\author{Sixuan Gu}
\address{Institute of Mathematical Sciences, The Chinese University of Hong Kong, Hong Kong, China}
\email{sixuangu@link.cuhk.edu.cn}

\author{Wei Qi\,\orcidlink{0009-0004-5794-4907}}
\address{Department of Physics, The Ohio State University, Columbus, OH 43210, USA}
\email{qi.673@osu.edu}

\author{Yaoyu Cheng}
\address{Faculty of Business Administration, The Chinese University of Hong Kong, Hong Kong, China}
\email{yaoyucheng@link.cuhk.edu.cn}

\date{}

\begin{document}

\maketitle

\begin{abstract}
Let $R$ be a commutative semiring, not necessarily with a multiplicative identity, and let $A\in M_n(R)$. We prove that the map $x\mapsto Ax$ on $R^n$ is injective if and only if $x\mapsto A^T x$ is injective. Equivalently, the left- and right-cancellative elements of the multiplicative semigroup $M_n(R)$ coincide. The proof splits formal determinant expansions into their even and odd halves; it uses no subtraction, additive cancellation, group completion, inverse, or multiplicative identity. As consequences we recover the stable-finiteness theorem for matrices over unital commutative semirings. We also prove that surjectivity is invariant under transpose. In fact, the existence of a surjective square matrix of positive size forces $R$ to have a multiplicative identity.
\end{abstract}

\noindent\textbf{Keywords.} Commutative semiring; matrix semigroup; injective linear map; transpose; cancellative element; stable finiteness.

\noindent\textbf{2020 Mathematics Subject Classification.} Primary 16Y60; Secondary 15A15, 20M20.

\section{Introduction}

Throughout, a \emph{commutative semiring} means a commutative monoid $(R,+,0)$ together with an associative and commutative multiplication that distributes over addition and for which $0$ is absorbing. A multiplicative identity is not assumed. A matrix acts by the usual finite sums of products. Coordinate relabelings are understood simply as reindexings of tuples; permutation matrices are never used.

Over a field, and more generally over a commutative ring, injectivity of a square matrix and of its transpose are governed by the same determinant. The same conclusion follows for a unital additively cancellative commutative semiring by passage to its Grothendieck ring. Neither argument survives for a general semiring: equality $Ax=Ay$ cannot be converted into a nonzero vector in a kernel, and the two parity halves of a determinant expansion cannot be subtracted. In particular, injectivity is stronger than having a trivial zero-kernel. Over the Boolean semiring, for example,
\[
\begin{pmatrix}
1&1\\
1&1
\end{pmatrix}
\begin{pmatrix}
1\\
0
\end{pmatrix}
=
\begin{pmatrix}
1&1\\
1&1
\end{pmatrix}
\begin{pmatrix}
0\\
1
\end{pmatrix},
\]
although the zero-kernel of the matrix is trivial. Thus partial results about left and right zero divisors under additional entry hypotheses, such as those in \cite{Kanan2016}, do not in general control injectivity as reflection of equality; the cited paper also records a failure of the converse zero-divisor implication.

There are several related left--right symmetry results. Reutenauer and Straubing proved that every unital commutative semiring is stably finite: each matrix semiring $M_n(R)$ is directly finite, so $AB=I_n$ implies $BA=I_n$ \cite{ReutenauerStraubing1984}; see \cite{Shitov2025} for a short proof based on positive and negative determinant expansions. Poplin and Hartwig developed a systematic calculus of determinantal identities over commutative semirings \cite{PoplinHartwig2004}. Row--column duality and matrix semigroups over semirings have also been studied in \cite{WildingJohnsonKambites2013} and \cite{GouldJohnsonNaz2020}; for earlier work on matrices over semirings see \cite{Ghosh1996}. Transposition is, of course, an involutory anti-isomorphism of $M_n(R)$, but by itself it only exchanges a property of $A$ with the opposite-sided property of $A^T$; it does not compare the two sides for the same matrix.

The question whether injectivity itself is transpose-symmetric was posed explicitly in \cite{Xu2026}. That question records elementary positive answers in dimension two and for rings, additively cancellative semirings, and finite unital semirings, but states that there is no existing result on the $3\times3$ case. To the best of our knowledge, the following theorem gives the first and positive answer in all finite dimensions for semirings without additive cancellation or a multiplicative identity.

A \texttt{Lean 4} formalization for this paper, using \texttt{mathlib} \cite{lean4,mathlib}, is available at \url{https://github.com/wqirocks/transpose-injectivity}.

\begin{theorem}
Let $R$ be a commutative semiring and let $A\in M_n(R)$, where $n\geq1$. Then
\[
A:R^n\longrightarrow R^n\quad\text{is injective}
\quad\Longleftrightarrow\quad
A^T:R^n\longrightarrow R^n\quad\text{is injective}.
\]
\end{theorem}

There is an equivalent semigroup formulation. An element $A\in M_n(R)$ is \emph{left cancellative} if $AB=AC$ implies $B=C$, and \emph{right cancellative} if $BA=CA$ implies $B=C$. Acting on columns shows that $A$ is left cancellative exactly when $x\mapsto Ax$ is injective. Transposition shows that $A$ is right cancellative exactly when $x\mapsto A^T x$ is injective. Thus Theorem 1 says precisely that the left- and right-cancellative elements of $M_n(R)$ coincide. This may be viewed as the cancellation analogue of the stable-finiteness theorem of Reutenauer--Straubing.

The proof uses two separation arguments. First, parity-minor symmetrizers separate each summand $a_{pq}x_q$ in an equality $Ax=Ay$. Second, balanced minor relations descend in size and turn equality after multiplication by one column into equality of the original scalars. The parity calculus is set up in Section 2, the two separation arguments occupy Sections 3 and 4, and Section 5 gives the proof and consequences concerning one-sided inverses and surjectivity.

\section{Parity notation}

We first introduce notation that replaces signed determinants. If $M=(m_{ij})$ is an ordered $r\times r$ matrix and $e\in\mathbb{Z}/2\mathbb{Z}$, put
\[
\Delta_e(M)=\sum_{\substack{\pi\in S_r\\\parity(\pi)=e}}\prod_{i=1}^r m_{i,\pi(i)}.
\]
Thus $\Delta_0(M)$ and $\Delta_1(M)$ are respectively the even and odd halves of the formal determinant expansion. All parity subscripts are read in $\mathbb{Z}/2\mathbb{Z}$.

For ordered row and column tuples $I,J$ of the same length, $A_{I,J}$ denotes the corresponding ordered submatrix. Deleting an entry from an ordered tuple preserves the order of the remaining entries. Unless distinctness is explicitly imposed, an ordered tuple may contain repeated indices; this is essential for the repeated-column identities below.

We shall use the following three polynomial identities.

\begin{lemma}[Parity identities]
Let $M$ be an ordered square matrix of positive size.
\begin{enumerate}
\renewcommand{\labelenumi}{(\roman{enumi})}
\item If $M'$ is obtained from $M$ by interchanging two columns, then
\[
\Delta_e(M')=\Delta_{e+1}(M).
\]
\item If $M$ has two equal columns, then
\[
\Delta_0(M)=\Delta_1(M).
\]
\item If $r\geq2$, paritywise Laplace expansion in the last column of an $r\times r$ matrix $M$ gives
\[
\Delta_e(M)=\sum_{t=1}^r m_{tr}\Delta_{e+t+r}(M_{\hat t,\hat r}).
\]
\end{enumerate}
\end{lemma}

\begin{proof}
Part (i) follows by composing permutations with the transposition of the two columns. If the two columns are equal, this gives a product-preserving bijection between the even and odd summands, proving (ii). For (iii), a permutation using the entry in row $t$ of the last column has parity equal to the parity of its complementary permutation plus $t+r$.
\end{proof}

No subtraction in $R$ was used in the process. Parity is only a label on two finite sums of monomials.

\section{Separation of individual matrix summands}

We now assume $n\geq2$. The difficult step is to separate every summand in a matrix equality without cancellation.

\begin{lemma}[Summand separation]
Assume that $A^T$ is injective. If $Ax=Ay$, then
\[
a_{pq}x_q=a_{pq}y_q\qquad(1\leq p,q\leq n).
\]
\end{lemma}

\begin{proof}
Fix $p,q$. Define $\delta_j\in\mathbb{Z}/2\mathbb{Z}$ by
\[
\delta_j=
\begin{cases}
1,&j=q,\\
0,&j\neq q.
\end{cases}
\]
For an ordered $m$-tuple $I$ of distinct rows and an ordered $(m-1)$-tuple $J$ of distinct columns, with $q\notin J$, define
\begin{equation}
L^e_{I,J}(v)=\sum_{j=1}^n\Delta_{e+\delta_j}\bigl(A_{I,(J,j)}\bigr)v_j,
\qquad e\in\mathbb{Z}/2\mathbb{Z}.
\end{equation}
The appended column $j$ is allowed to repeat a column already in $J$. We claim, simultaneously for both $e$, that
\begin{equation}
L^e_{I,J}(x)=L^e_{I,J}(y).
\end{equation}
whenever $1\leq m\leq n-1$. We prove this by downward induction on $m$.

\emph{Initial step:} $m=n-1$. Let $\rho$ be the row missing from $I$ and order all rows as
\[
H=(I,\rho)=(h_1,\ldots,h_n).
\]
For fixed $I,J,e$, define an $n\times n$ matrix $B$ by
\begin{equation}
B_{h_tj}=\Delta_{e+\delta_j+t+n}\bigl(A_{H\setminus h_t,(J,j)}\bigr).
\end{equation}
Paritywise Laplace expansion gives
\begin{equation}
(A^T B)_{ij}=\Delta_{e+\delta_j}\bigl(A_{H,(J,j,i)}\bigr).
\end{equation}
We claim that this matrix is symmetric. If the columns in (3.4) are all distinct, then $i$ and $j$ are the two columns complementary to $J$. Since $q\notin J$, exactly one of $i,j$ equals $q$. Interchanging the last two columns changes the parity by one, while $\delta_i=\delta_j+1$. Hence
\[
\Delta_{e+\delta_j}\bigl(A_{H,(J,j,i)}\bigr)
=
\Delta_{e+\delta_i}\bigl(A_{H,(J,i,j)}\bigr).
\]
If a column is repeated, interchanging the last two positions swaps the parity halves by Lemma 2(i), while Lemma 2(ii) makes those two halves identical. The $\delta$-shift is therefore also immaterial, and the same equality follows. Consequently
\begin{equation}
A^T B=B^T A.
\end{equation}
Since $Ax=Ay$, equation (3.5) gives
\[
A^T Bx=B^T Ax=B^T Ay=A^T By.
\]
Injectivity of $A^T$ yields $Bx=By$. In (3.3), row $\rho=h_n$ is precisely the coefficient row in (3.1), because $n+n=0$ in $\mathbb{Z}/2\mathbb{Z}$. Thus its row equality is (3.2).

\emph{Downward step.} Suppose (3.2) is known for row tuples of length $m+1$. Fix $I,J$ at length $m$, choose a row $\tau\notin I$, and put
\[
K=(I,\tau)=(k_1,\ldots,k_{m+1}).
\]
For $v\in R^n$ define a vector $U(v)$, supported on the rows in $K$, by
\begin{equation}
U_{k_t}(v)=\sum_{j=1}^n\Delta_{e+\delta_j+t+m+1}\bigl(A_{K\setminus k_t,(J,j)}\bigr)v_j.
\end{equation}
Thus $U_r(v)=0$ for every row $r\notin K$. Another paritywise Laplace expansion gives, for every column $z$,
\begin{equation}
(A^T U(v))_z=\sum_{j=1}^n\Delta_{e+\delta_j}\bigl(A_{K,(J,j,z)}\bigr)v_j.
\end{equation}
If $z\in J$, every matrix on the right has a repeated column, so the $\delta_j$ shift is immaterial. The right side of (3.7) is then
\begin{equation}
\Delta_e\bigl(A_{K,(J,Av,z)}\bigr),
\end{equation}
where the notation means that the penultimate column is the restriction of $Av$ to the rows $K$. If $z=q$, the only shifted term has $j=q$, and that term has two equal $q$-columns. Thus (3.8) again holds. In both cases the value is the same for $v=x$ and $v=y$ because $Ax=Ay$.

Finally, if $z\notin J\cup\{q\}$, interchanging the final two columns shows that the right side of (3.7) is
\[
L^{e+1}_{K,(J,z)}(v),
\]
which is equal for $x,y$ by the induction hypothesis. Hence
\[
A^T U(x)=A^T U(y).
\]
Injectivity gives $U(x)=U(y)$. In (3.6), the component corresponding to $\tau=k_{m+1}$ is exactly $L^e_{I,J}$, because $(m+1)+(m+1)=0$ in $\mathbb{Z}/2\mathbb{Z}$. This proves (3.2) and completes the downward induction.

Finally take $m=1$, $I=(p)$, let $J$ be the empty tuple, and take $e=1$. For a $1\times1$ matrix $[a]$ one has $\Delta_0([a])=a$ and $\Delta_1([a])=0$. Therefore (3.1) reduces to
\[
L^1_{(p),\varnothing}(v)=a_{pq}v_q.
\]
Equation (3.2) is consequently the desired equality $a_{pq}x_q=a_{pq}y_q$.
\end{proof}

\begin{remark}[What happens for $4\times4$ matrices]
When $n=4$, the initial matrices $B$ in (3.3) have entries that are parity halves of $3\times3$ cofactors, hence homogeneous cubic polynomials in the entries of $A$. The identity $A^T B=B^T A$ first produces the weighted $3\times3$ minor relations. Two downward steps, from $m=3$ to $m=2$ and then to $m=1$, isolate $a_{pq}x_q$.
\end{remark}

\section{Balanced minors and one-column separation}

The next lemma turns the product equalities from Lemma 3 into equality of scalars.

For $r,s\in R$ and an ordered square matrix $M$, write
\begin{equation}
\Bal_{r,s}(M)\quad\text{for the equality}\quad
\Delta_0(M)r+\Delta_1(M)s=\Delta_0(M)s+\Delta_1(M)r.
\end{equation}

\begin{lemma}[One-column separation]
Let $n\geq2$ and assume that $A^T$ is injective. Fix a column $q$ and let $r,s\in R$. If
\begin{equation}
a_{iq}r=a_{iq}s\qquad(1\leq i\leq n),
\end{equation}
then $r=s$.
\end{lemma}

\begin{proof}
Every monomial in either parity half of a full $n\times n$ expansion contains exactly one entry from column $q$. Hence (4.2), multiplied by the other factors in the monomial, shows termwise that every full-size ordered minor satisfies (4.1).

We show that balanced relations descend from size $m$ to size $m-1$, for $2\leq m\leq n$. Assume that (4.1) holds for every ordered size-$m$ minor. Let $I,J$ be ordered tuples of $m-1$ distinct rows and columns; we prove the relation for $A_{I,J}$. Choose a row $\sigma\notin I$ and write
\[
K=(I,\sigma)=(k_1,\ldots,k_m).
\]
For $t=1,\ldots,m$ and $e\in\mathbb{Z}/2\mathbb{Z}$, put
\begin{equation}
C_{t,e}=\Delta_{e+t+m}\bigl(A_{K\setminus k_t,J}\bigr).
\end{equation}
Define vectors $u,v\in R^n$, supported on the rows in $K$, by
\begin{equation}
u_{k_t}=C_{t,0}r+C_{t,1}s,
\qquad
v_{k_t}=C_{t,0}s+C_{t,1}r.
\end{equation}
For a column $\ell\notin J$, paritywise Laplace expansion shows that the $\ell$-coordinate equality in $A^T u=A^T v$ is precisely
\[
\Bal_{r,s}\bigl(A_{K,(J,\ell)}\bigr),
\]
which holds by the size-$m$ hypothesis. If $\ell\in J$, the corresponding matrix has two equal columns, so $\Delta_0=\Delta_1$ and the two sides are the same two summands in opposite order. Therefore $A^T u=A^T v$. Injectivity gives $u=v$. Taking the component belonging to $\sigma=k_m$, equation (4.3) becomes
\[
C_{m,0}=\Delta_0(A_{I,J}),
\qquad
C_{m,1}=\Delta_1(A_{I,J}),
\]
so that component equality is $\Bal_{r,s}(A_{I,J})$.

Descending to size one gives
\begin{equation}
a_{ij}r=a_{ij}s\qquad(1\leq i,j\leq n).
\end{equation}
Fix a row $i$. Let $u'$ and $v'$ have respectively $r$ and $s$ in component $i$ and zero in every other component. Equation (4.5) says $A^T u'=A^T v'$. Injectivity yields $u'=v'$, and hence $r=s$.
\end{proof}

\section{Proof of the main theorem and consequences}

\begin{proof}[Proof of Theorem 1]
The case $n=1$ is immediate. Suppose $n\geq2$ and that $A^T$ is injective. If $Ax=Ay$, Lemma 3 gives, for every fixed column $q$,
\[
a_{iq}x_q=a_{iq}y_q\qquad(1\leq i\leq n).
\]
Lemma 5, applied with $r=x_q$ and $s=y_q$, gives $x_q=y_q$. This holds for every $q$, so $x=y$. Thus $A$ is injective.

Thus we have proved, for every $A$,
\[
A^T\text{ injective}\qquad\Longrightarrow\qquad A\text{ injective}.
\]
Applying this to $A^T$ gives the converse.
\end{proof}

\begin{remark}[Nature of the identities]
Every unconditional algebraic identity used above is an equality between two sums of monomials in the free commutative semiring. The other equalities follow only by substitution from the stated hypotheses and by injectivity of $A^T$. No common summand is ever cancelled. In particular, the proof remains valid for noncancellative and nonunital commutative semirings.
\end{remark}

\subsection{Stable finiteness}

The main theorem immediately recovers the one-sided inverse theorem of Reutenauer and Straubing.

\begin{corollary}[Reutenauer--Straubing]
Every unital commutative semiring is stably finite. Equivalently, if $A,B\in M_n(R)$ and $AB=I_n$, then $BA=I_n$.
\end{corollary}

\begin{proof}
Transposing $AB=I_n$ gives $B^T A^T=I_n$, so $A^T$ is injective. Theorem 1 makes $A$ injective. Since
\[
A(BA)=(AB)A=A=AI_n,
\]
injectivity applied columnwise gives $BA=I_n$.
\end{proof}

\subsection{Surjectivity}

We next prove the surjective analogue, including for semirings not initially assumed to have an identity. For $r\in R$, let $D_n(r)$ denote the matrix with $r$ on its diagonal and $0$ elsewhere.

\begin{lemma}[Scaled reversal]
Let $n\geq2$ and let $A,B\in M_n(R)$. Suppose
\[
AB=\diag(t_1,\ldots,t_n).
\]
Set
\[
T=\prod_{k=1}^n t_k,
\qquad
T_s=\prod_{\substack{k=1\\k\neq s}}^n t_k,
\qquad
Y_{is}=b_{is}T_s.
\]
Then $YA=D_n(T)$.
\end{lemma}

\begin{proof}
Put $C=AB$. If $C^{(r|s)}$ denotes the ordered matrix obtained from $C$ by deleting row $r$ and column $s$, separating the positive and negative monomials in the classical adjugate polynomial identity gives the following identity
\begin{equation}
\begin{aligned}
&\sum_{r,s=1}^n b_{is}a_{rj}\Delta_{r+s}\bigl(C^{(r|s)}\bigr)
+
\begin{cases}
\Delta_1(C),&i=j,\\
0,&i\neq j
\end{cases}
\\
&\qquad=
\sum_{r,s=1}^n b_{is}a_{rj}\Delta_{r+s+1}\bigl(C^{(r|s)}\bigr)
+
\begin{cases}
\Delta_0(C),&i=j,\\
0,&i\neq j.
\end{cases}
\end{aligned}
\end{equation}
Here the parity subscripts are taken in $\mathbb{Z}/2\mathbb{Z}$, and the two cases on each side are external case distinctions rather than multiplication by a scalar Kronecker delta. For completeness, form the identity first in the polynomial ring $\mathbb{Z}[(a_{ij}),(b_{ij})]$. There
\[
B\adj(AB)A=\det(AB)I_n.
\]
Expanding and moving every negative monomial to the opposite side gives exactly (5.1). All of its monomials have positive degree, so the resulting equality of $\mathbb{N}$-polynomials may be evaluated in a nonunital semiring. This is the parity-split adjugate mechanism used in \cite{ReutenauerStraubing1984,Shitov2025}.

Now $C=\diag(t_1,\ldots,t_n)$. If $r\neq s$, the complementary matrix $C^{(r|s)}$ has a zero row or a zero column, so both parity halves vanish. If $r=s$, its even half is $T_s$ and its odd half is $0$. Likewise, $\Delta_0(C)=T$ and $\Delta_1(C)=0$. Thus (5.1) reduces to
\[
\sum_{s=1}^n b_{is}T_s a_{sj}=
\begin{cases}
T,&i=j,\\
0,&i\neq j,
\end{cases}
\]
which is precisely $YA=D_n(T)$. Since $n\geq2$, every product occurring here has positive length; no empty product or multiplicative identity has been used.
\end{proof}

\begin{theorem}[Transpose symmetry of surjectivity]
Let $R$ be a commutative semiring and $A\in M_n(R)$, where $n\geq1$. Then $A:R^n\to R^n$ is surjective if and only if $A^T:R^n\to R^n$ is surjective.
\end{theorem}

\begin{proof}
The case $n=1$ is immediate because $A=A^T$. Assume $n\geq2$ and that $A$ is surjective. For $m\geq1$, let $P_m(R)$ be the additive submonoid of $R$ consisting of finite sums of products of exactly $m$ elements of $R$. Lifting a vector with an arbitrary $r\in R$ in one coordinate and $0$ elsewhere shows that $R=P_2(R)$. If $R=P_m(R)$ and $r=\sum_\gamma u_\gamma v_\gamma$ is a two-factor decomposition, expanding each $u_\gamma$ as a sum of $m$-fold products shows by distributivity that $r\in P_{m+1}(R)$. Hence $R=P_m(R)$ for every $m\geq2$.

Given arbitrary $t_1,\ldots,t_n\in R$, surjectivity allows us to choose the columns of a matrix $B$ so that
\[
AB=\diag(t_1,\ldots,t_n).
\]
Lemma 8 supplies a matrix $Y$ with $YA=D_n(t_1\cdots t_n)$. Now let $r\in R=P_n(R)$ and write
\[
r=\sum_\alpha\prod_{k=1}^n t_{\alpha k}.
\]
Adding the matrices $Y$ associated with these products gives a matrix $Y_r$ satisfying
\begin{equation}
Y_rA=D_n(r),
\qquad\text{and hence}\qquad
A^T Y_r^T=D_n(r).
\end{equation}

Let $z=(z_1,\ldots,z_n)^T\in R^n$. Since $R=P_2(R)$, write
\[
z_i=\sum_\beta r_{i\beta}s_{i\beta}\qquad(1\leq i\leq n).
\]
For each $i,\beta$, let $v_{i\beta}$ have $s_{i\beta}$ in coordinate $i$ and $0$ in every other coordinate. By (5.2),
\[
A^T\bigl(Y_{r_{i\beta}}^T v_{i\beta}\bigr)=D_n(r_{i\beta})v_{i\beta},
\]
whose sole possibly nonzero coordinate is $r_{i\beta}s_{i\beta}$ in position $i$. Summing these preimages gives a preimage of $z$ under $A^T$. Hence $A^T$ is surjective. Applying the same implication to $A^T$ proves the converse.
\end{proof}

\begin{corollary}
If a commutative semiring admits a surjective square matrix of positive size, then it has a multiplicative identity. Every such surjective matrix is invertible.
\end{corollary}

\begin{proof}
Suppose $A\in M_n(R)$ is surjective. By Theorem 9, $A^T$ is surjective. Choose a matrix $C$ columnwise so that $A^T C=A^T$, and put $E=C^T$. Then $EA=A$. Since $A$ is surjective, $E$ acts as the identity on every vector of $R^n$. Applying this to a vector with an arbitrary $r\in R$ in coordinate $i$ and $0$ elsewhere gives $e_{ii}r=r$. Commutativity makes $e_{ii}$ a two-sided multiplicative identity for $R$.

Now that $R$ is known to be unital, choose $B$ columnwise with $AB=I_n$. Corollary 7 gives $BA=I_n$, so $A$ is invertible.
\end{proof}
\section*{Acknowledgements}
The author Sixuan Gu is simultaneously supervised by Professor Jiu-Kang Yu of Hetao Institute of Mathematics and Interdisciplinary Sciences and Professor Caihua Luo of the Chinese University of Hong Kong (Shenzhen) during this research. Professor Yu gave valuable advice on generalizing the theorem from dimension 3 to arbitrary dimensions, and to surjectivity of the linear maps. 
\section*{Declaration of generative AI and AI-assisted technologies in the manuscript preparation process}

During the preparation of this work, the author used OpenAI's ChatGPT to explore proof strategies, check formal polynomial identities, and assist in drafting. The author reviewed and edited the resulting material and takes full responsibility for the content of the publication.


\begin{thebibliography}{99}

\bibitem{Kanan2016}
A. M. Kanan, Zero divisors for matrices over commutative semirings,
\emph{ScienceAsia} \textbf{42} (2016), 362--365.
doi:10.2306/scienceasia1513-1874.2016.42.362.

\bibitem{ReutenauerStraubing1984}
C. Reutenauer and H. Straubing, Inversion of matrices over a commutative semiring,
\emph{J. Algebra} \textbf{88} (1984), no. 2, 350--360.
doi:10.1016/0021-8693(84)90070-X.

\bibitem{Shitov2025}
Y. Shitov, Easy matrix inversion in semirings, preprint (2025).
doi:10.13140/RG.2.2.33126.54084.

\bibitem{PoplinHartwig2004}
P. L. Poplin and R. E. Hartwig, Determinantal identities over commutative semirings,
\emph{Linear Algebra Appl.} \textbf{387} (2004), 99--132.
doi:10.1016/j.laa.2004.02.019.

\bibitem{WildingJohnsonKambites2013}
D. Wilding, M. Johnson, and M. Kambites, Exact rings and semirings,
\emph{J. Algebra} \textbf{388} (2013), 324--337.
doi:10.1016/j.jalgebra.2013.05.005.

\bibitem{GouldJohnsonNaz2020}
V. Gould, M. Johnson, and M. Naz, Matrix semigroups over semirings,
\emph{Internat. J. Algebra Comput.} \textbf{30} (2020), no. 2, 267--337.
doi:10.1142/S0218196720500010.

\bibitem{Ghosh1996}
S. Ghosh, Matrices over semirings,
\emph{Inform. Sci.} \textbf{90} (1996), 221--230.
doi:10.1016/0020-0255(95)00283-9.

\bibitem{Xu2026}
J. Xu, Transpose symmetry of injectivity of linear maps over semirings,
MathOverflow question 511862 (2026).
\url{https://mathoverflow.net/questions/511862/}.

\bibitem{lean4}
L. de Moura and S. Ullrich, The Lean 4 theorem prover and programming language,
in \emph{Automated Deduction---CADE 28},
Lecture Notes in Computer Science, vol. 12699,
Springer, 2021, pp. 625--635.
doi:10.1007/978-3-030-79876-5\_37.

\bibitem{mathlib}
The mathlib Community, The Lean mathematical library,
in \emph{Proceedings of the 9th ACM SIGPLAN International Conference
on Certified Programs and Proofs (CPP 2020)},
ACM, 2020, pp. 367--381.
doi:10.1145/3372885.3373824.

\end{thebibliography}
\end{document}